\documentclass[graybox]{svmult}

\usepackage{type1cm}        
\usepackage{makeidx}         
\usepackage{graphicx}        
\usepackage{multicol}        
\usepackage[bottom]{footmisc}

\usepackage{newtxtext}       %

\usepackage[utf8]{inputenc}
\usepackage[T1]{fontenc}
\usepackage{tikz}
\usepackage{pgfplots}
\usepackage{amsmath}
\usepackage{amssymb}

\newcommand{\R}{\mathbb{R}}
\newcommand{\N}{\mathbb{N}}

\newcommand{\Rpos}{\mathbb{R}_{+}}

\newcommand{\diam}{\operatorname*{diam}}

\newcommand{\abs}[1]{\left\vert #1 \right\vert}
\newcommand{\skp}[1]{\left< #1 \right>}
\newcommand{\norm}[1]{\left\| #1 \right\|}

\newcommand{\dist}{\operatorname*{dist}}
\renewcommand{\div}{\operatorname*{div}}

\makeindex             
\begin{document}

\title*{Exponential convergence of $hp$-FEM for the integral fractional Laplacian in 1D}
\author{Markus Faustmann \and Carlo Marcati \and Jens Markus Melenk \and Christoph Schwab}
\institute{Markus Faustmann, Jens Markus Melenk  \at Institute of Analysis and Scientific Computing, TU Wien, Vienna, Austria 
\email{\{markus.faustmann, melenk\}@tuwien.ac.at} \newline
Carlo Marcati, Christoph Schwab 
\at Seminar for Applied Mathematics, ETH Z\"urich, Z\"urich, Switzerland 
\email{\{carlo.marcati, schwab\}@sam.math.ethz.ch}
}%
%
%
\maketitle


\abstract{We prove weighted analytic regularity for the solution 
of the integral fractional Poisson problem on bounded intervals 
with analytic right-hand side. Based on this regularity result,
we prove exponential convergence of the $hp$-FEM 
on geometric boundary-refined meshes. 
}

\section{Introduction}
The (numerical) analysis of non-integer powers of elliptic differential operators has garnered a lot of interest recently, as such operators can be used to derive models for anomalous diffusion processes in various applications. 
The prototype of such an operator is the so called fractional Laplacian $(-\Delta)^s$ for $s \in (0,1)$, which can be defined in several different ways. 
On the full space, a classical way is to define it as an operator with Fourier symbol $|\xi|^{2s}$, but (equivalent) alternatives via semi-group theory, operator theory, or via singular integrals exist as well,~\cite{Kwasnicki}.  Here, we consider the so called integral fractional Laplacian, using the singular integral definition.

Solutions to equations involving the integral fractional Laplacian on bounded domains usually are non-smooth at the boundary and feature only finite regularity
even if the data is smooth, \cite{Grubb15};  
we mention finite regularity results in Hölder spaces \cite{ros-oton-serra14,Grubb15} (smooth domains), in isotropic weighted Sobolev spaces \cite{acosta-borthagaray17} or unweighted Besov spaces \cite{BN21} (Lipschitz domains). In our recent work, \cite{FMMS21}, we provide weighted analytic regularity estimates in two space dimensions that reflect 
both the interior analyticity and the anisotropic nature of the solution near the boundary. 
In the present article, we provide similar weighted analytic estimates for the one dimensional case.

The numerical approximation of fractional PDEs by means of finite element techniques is in principle well understood; we mention
the survey articles \cite{GunbActa,BBNOS18,BLN19,LPGSGZMCMAK18} and references therein. An important active field is the design of suitably designed meshes to counteract the singular behavior at the boundary, \cite{acosta-borthagaray17,BBNOS18}. We mention \cite{GiEPSSt}, where a sharp analysis of vertex and edge 
singularities via Mellin techniques is used to derive the correct mesh grading 
for methods converging at the optimal algebraic rate. 
In this article, we leverage our weighted analytic regularity estimates to design an 
exponentially convergent method by means of $hp$-finite element approximation in one dimension.
The generalization to two dimensions is the topic of our follow-up work \cite{FMMS-hp}. 

The techniques employed in the present work are closely related to the higher dimensional case done in \cite{FMMS21}, 
but the one dimensional case allows for simplifications and 
clearer presentation of the main concepts. 
A crucial step is a reformulation of the fractional PDE as a Dirichlet to Neumann operator 
of a degenerate \emph{local} elliptic PDE, the so-called
Caffarelli-Silvestre extension, \cite{CafSil07}. 
For this extension, 
the second crucial step is that a (global) regularity shift of essentially $1/2$ 
can be obtained using difference quotient techniques, \cite{Savare}. 
Then, 
we derive interior regularity estimates of Caccioppoli type and bootstrap these regularity results 
to obtain weighted analytic regularity.
%
\section{Model problem and main results}
\subsection{The fractional Laplacian}
We consider the integral fractional Laplacian defined 
pointwise as the principal value integral  
\begin{align*}
(-\Delta)^su(x) 
:= 
C(s) \; \text{P.V.} \int_{\R}\frac{u(x)-u(z)}{\abs{x-z}^{1+2s}} \, dz 
\quad \text{with} \quad
C(s):= - 2^{2s}\frac{\Gamma(s+1/2)}{\pi^{1/2}\Gamma(-s)},
\end{align*}
where $\Gamma(\cdot)$ denotes the Gamma function. 

Fractional differential equations are conveniently described using fractional Sobolev spaces.
Let $\Omega \subset \R$ be an open, bounded interval. 
Denoting by $H^t(\Omega)$ the classical integer order Sobolev spaces with $t \in \N_0$, 
fractional order Sobolev spaces for $t \in (0,1)$ are 
defined in terms of the Aronstein-Slobodeckij seminorm 
$|\cdot|_{H^t(\Omega)}$ 
and the corresponding full norm $\|\cdot\|_{H^t(\Omega)}$,
which are given by
\begin{align*}
|v|^2_{H^t(\Omega)} 
= 
\int_{\Omega} \int_{\Omega} \frac{|v(x) - v(z)|^2}{\abs{x-z}^{1+2t}}
\,dz\,dx, 
\qquad 
\|v\|^2_{H^t(\Omega)} = \|v\|^2_{L^2(\Omega)} + |v|^2_{H^t(\Omega)}.
\end{align*}
For $t \in (0,1)$, we employ the spaces 
$\widetilde{H}^{t}(\Omega) := \left\{u \in H^t(\R^d) : u\equiv 0 \; \text{on} \; \R^d \backslash \overline{\Omega} \right\}$ 
with norm 
\begin{align*}
\;  \norm{v}_{\widetilde{H}^{t}(\Omega)}^2 := \norm{v}_{H^t(\Omega)}^2 + \norm{v/r^t}_{L^2(\Omega)}^2,
\end{align*}
where $r(x):=\operatorname{dist}(x,\partial\Omega)$ 
 denotes the Euclidean distance of a point
 $x \in \Omega$ from the boundary $\partial \Omega$.
%
For $t \in (0,1)\backslash \{\frac 1 2\}$, the norms $\norm{\cdot}_{\widetilde{H}^{t}(\Omega)}$  and $\norm{\cdot}_{H^{t}(\Omega)}$ are equivalent, 
see, e.g., \cite{Grisvard}.
For $t > 0$, the space $H^{-t}(\Omega)$ denotes the dual space of $\widetilde{H}^t(\Omega)$, 
and we write $\skp{\cdot,\cdot}_{L^2(\Omega)}$ 
for the duality pairing that extends the $L^2(\Omega)$-inner product.

On a bounded interval $\Omega \subset \R$, we consider 
the fractional differential equation 
\begin{align}\label{eq:modelproblem}
 (-\Delta)^su &= f \qquad \text{in}\, \Omega, \\
\label{eq:modelproblem-bc}
 u &= 0 \quad \quad\, \text{in}\, \Omega^c:=\R \backslash \overline{\Omega},
\end{align}
%
where $s \in (0,1)$ and $f \in H^{-s}(\Omega)$ is a given right-hand side. 
The weak form of (\ref{eq:modelproblem}), (\ref{eq:modelproblem-bc}) 
is to find $u \in \widetilde{H}^s(\Omega)$ such that 
\begin{equation}
\label{eq:weakform}
a(u,v):= \frac{C(s)}{2} \int_\R\int_{\R} 
 \frac{(u(x)-u(z))(v(x)-v(z))}{\abs{x-z}^{1+2s}} \, dz \, dx = \skp{f,v}_{L^2(\Omega)} 
\end{equation}
for all $v \in \widetilde{H}^s(\Omega)$.
Existence and uniqueness of $u \in \widetilde{H}^s(\Omega)$ follow from
the Lax--Milgram Lemma for any $f \in H^{-s}(\Omega)$, upon the observation
that the bilinear form $a(\cdot,\cdot): \widetilde{H}^s(\Omega)\times \widetilde{H}^s(\Omega)\to \R$ 
is continuous and coercive, 
see, e.g., \cite[Sec.~{2.1}]{acosta-borthagaray17}.
\subsection{Weighted analytic regularity}
\label{sec:WgtAnReg}
Our first main result, Theorem \ref{thm:analytic-regularity}, 
asserts analytic regularity of the solution $u\in \widetilde{H}^s(\Omega)$ 
to our model problem \eqref{eq:weakform} in scales of weighted Sobolev spaces, 
provided $f$ in \eqref{eq:weakform} is analytic in $\overline{\Omega}$.
The weights for these spaces are powers of the distance to the boundary of 
the computational domain, $r(x):= \operatorname{dist}(x,\partial\Omega)$.

\begin{theorem}
\label{thm:analytic-regularity}
Let the data $f\in C^\infty(\overline{\Omega})$ satisfy, for constants $C_f$, $\gamma_f >0$, 
\begin{align*}
\forall p \in \N_0\colon \quad 
\norm{D^p f}_{L^2(\Omega)} \leq C_f\gamma_f^p p!. 
\end{align*}
Let $u$ solve \eqref{eq:weakform}.
Then, there is $\gamma$ (depending only on $\gamma_f$, $s$, $\Omega$) such that for any $\varepsilon>0$ 
there exists  a constant $C_{\varepsilon}>0$ (depending on $C_f$, $s$, $\Omega$, $\varepsilon$) 
such that 
\begin{align*}
\forall p \in \N\colon \quad 
\norm{r^{p-1/2-s+\varepsilon} D^p  u}_{L^2(\Omega)} \leq C_{\varepsilon} \gamma^p p!.
\end{align*}
In terms of the space 
${\mathcal B}^1_\beta := \{u \in L^2(\Omega)\,:\, \|r^{n + \beta} D^{n+1}u\|_{L^2(\Omega)}  \leq C \gamma^n n! \quad \forall n \in \N_0\}$
we have $u \in {\mathcal B}^1_\beta$ for $\beta = 1/2-s+\varepsilon$. In particular, $u \in C(\overline{\Omega})$. 
\end{theorem}

\subsection{Exponential convergence of $hp$-FEM}
\label{sec:1D-hpFEM}
Once weighted regularity results are available, numerical approximation by means of the $hp$-FEM, \cite{schwab98}, can be analyzed. 
In fact, employing geometric meshes and piecewise polynomials of higher degree, our second main result, Theorem~\ref{thm:hp-approx}, 
states exponential convergence of the $hp$-FEM for the integral fractional Laplacian.

\begin{definition}\label{def:geometric-mesh} {\bf (geometric mesh on $(-1,1)$) \;}
On $(-1,1)$, for a \emph{grading factor $\sigma \in (0,1)$} 
and a number \emph{$L$ of layers of geometric refinement}, 
the \emph{geometric mesh} ${\mathcal T}^{L}_{geo,\sigma} = 
\{T_i\colon i=1,\ldots,2L+2\}$ 
with $2L+2$ elements $T_i = (x_{i-1},x_{i})$ 
is given by the nodes
\begin{equation}\label{eq:GeoMes} 
\begin{array}{c}
x_0 := -1, \quad x_i = -1 + \sigma^{L-i+1}, \quad i=1,\ldots,L , 
\\
x_{i+1} = 1-\sigma^{i-L}, \quad i=L,\ldots,2L, \quad x_{2L+2} := 1. 
\end{array}
\end{equation} 
\end{definition}
Geometric partitions $\mathcal{T}^L_{geo,\sigma}$ on bounded intervals $\Omega \subset \R$ 
are obtained from \eqref{eq:GeoMes} by translation and dilation. 
Key features of a geometric partition 
${\mathcal T}^{L}_{geo,\sigma}$ are
a) elements $T_i \in \mathcal{T}^L_{geo,\sigma}$ with 
$\overline{T}_i \cap \partial\Omega  = \emptyset$ 
satisfy 
$\diam(T_i) \sim \dist(T_i,\partial\Omega)$ 
and 
b) elements $T_i \in \mathcal{T}^L_{geo,\sigma}$ with 
$\overline{T}_i \cap \partial\Omega \ne \emptyset$ satisfy 
$\diam(T_i)  = \mathcal{O}(\sigma^L)$. 
\medskip

On a geometric mesh $\mathcal{T}^{L}_{geo,\sigma}$ and for a polynomial degree $p \in \N$,
we introduce the spline space 
$S^{p,1}(\mathcal{T}^{L}_{geo,\sigma}) := \{v \in H^1(\Omega) : v|_{T_i} \in \mathcal{P}_p(T_i) \; \forall T_i \in \mathcal{T}^{L}_{geo,\sigma}\}$. 
Here, 
$\mathcal{P}_p(T_i)$ is the space of all polynomials of (at most) degree $p$ on $T_i$. 
The subspace with zero boundary conditions is
$S^{p,1}_0(\mathcal{T}^{L}_{geo,\sigma}) 
:= 
\{v \in S^{p,1}(\mathcal{T}^{L}_{geo,\sigma}) \; : \; v|_{\partial \Omega} = 0 \}$. 
We note 
$N:= \operatorname{dim} S^{p,1}_0(\mathcal{T}^{L}_{geo,\sigma}) \sim p L$. 

The $hp$-FEM approximation $u_N$ is the Galerkin discretization of 
\eqref{eq:weakform}:
\begin{align}\label{eq:hp-approx}
u_N \in S^{p,1}_0(\mathcal{T}^{L}_{geo,\sigma}):\quad 
a(u_N,v_N) =  \skp{f,v_N}_{L^2(\Omega)} \;\; 
\forall v_N \in S^{p,1}_0(\mathcal{T}^{L}_{geo,\sigma}).  
\end{align}

\begin{theorem}
\label{thm:hp-approx}
Let $\mathcal{T}^L_{geo,\sigma}$ be a geometric mesh on the interval $\Omega$ with grading factor $\sigma \in (0,1)$ and $L$ layers of refinement towards the boundary points. Let $u_N \in S^{p,1}_0(\mathcal{T}^{L}_{geo,\sigma})$ 
solve \eqref{eq:hp-approx} and $u$ solve \eqref{eq:weakform}. 
Then, there are constants $C_{\rm apx}$, $b > 0$ 
independent of $p$ and $L$  such that
\begin{align*}
 \norm{u-u_N}_{\widetilde H^s(\Omega)} \leq C_{\rm apx}(e^{-bp} + \sigma^{(1-\beta-s)L}),
\end{align*}
where $\beta \in (0,1)$ is given by Theorem~\ref{thm:analytic-regularity}. 
The particular choice $L \sim p$ leads to convergence 
 $\norm{u-u_N}_{\widetilde H^s(\Omega)} \leq C \exp(-b' \sqrt{N})$, 
where $N = \operatorname{dim} S^{p,1}_0({\mathcal T}^{L}_{geo,\sigma})$ 
is the problem size, and $C$, $b'$ are constants independent of $N$.
\end{theorem}
The rest of this note will provide short proofs of Theorems~\ref{thm:analytic-regularity} and~\ref{thm:hp-approx}.
\section{Regularity results}
%
\subsection{The Caffarelli-Silvestre extension}
\label{sec:CS-Ext}
The main tool in our regularity analysis is the very influential reformulation 
of the nonlocal fractional Laplacian as the Dirichlet-to-Neumann operator of a local,
degenerate elliptic PDE posed on a half space in one additional space dimension, 
the so-called \emph{Caffarelli-Silvestre extension}, \cite{CafSil07}.

For its formulation, let $\alpha := 1-2s$ 
and write $\omega^+:= \omega \times \R_+$ for any measurable subset $\omega \subset \R$, where $\R_+ = (0,\infty)$.
We define $L^2_\alpha(\omega^+)$  as the space of 
square-integrable functions with respect to the weight $y^\alpha$ with the norm 
\begin{align*}
 \norm{U}_{L^2_\alpha(\omega^+)}^2 := \int_{y \in \Rpos} y^{\alpha} \int_{x \in \omega} \abs{U(x,y)}^2 dx \, dy.
\end{align*}
Moreover, we introduce the Beppo-Levi space 
$H^1_{\alpha}(\R \times  \Rpos) := \{U \in L^2_{loc}(\R \times \Rpos)\,:\, 
\nabla U \in L^2_\alpha(\R \times \Rpos)\}$. 
For elements of $H^1_\alpha(\R \times \Rpos)$, one can give meaning to their trace 
at $y = 0$, which is denoted $\operatorname{tr} U$. 
In fact, 
$\operatorname{tr} U \in H^s(\R)$ (see, e.g., \cite[Lem.~3.8]{KarMel19}) with
$ \abs{\operatorname{tr} U}_{H^s(\R)} \lesssim \norm{\nabla U}_{L^2_\alpha(\R \times \Rpos)}$.
We also require the space 
$H^1_{\alpha,0}(\R \times \Rpos)
 :=
  \{ V\in  H^1_{\alpha}(\R \times \Rpos): \; \operatorname{tr} V = 0 \;\mbox{on}\; \Omega^c\}$.

Let data $F \in C^\infty_0(\R^{2})$ and $f \in C^\infty(\overline{\Omega})$ be given. 
The Caffarelli-Silvestre extension problem reads: Find $U = U(x,y) \in H^1_\alpha(\R \times \Rpos)$ such that
\begin{align}\label{eq:extension}
\nonumber
 -\div (y^\alpha \nabla U) &= F  &&\text{in} \; \R \times (0,\infty), 
\\ 
 \partial_{n_\alpha} U(\cdot,0) & = f  &&\text{in} \; \Omega, 
\\
\nonumber
\operatorname{tr} U & = 0 &&\text{on $\Omega^c$},
\end{align}
where $\partial_{n_\alpha} U(x,0) = - d_s\lim_{y \rightarrow 0}  y^\alpha \partial_y U(x,y)$
for $d_s = 2^{2s-1}\Gamma(s)/\Gamma(1-s)$.
The weak form of (\ref{eq:extension}) is: Find $U \in H^1_{\alpha,0}(\R \times \Rpos)$ such that 
for all $V \in H^1_{\alpha,0}(\R \times \R_+)$ 
\begin{align}
b(U,V) 
:= \int_{\R \times \R_+} y^\alpha \nabla U \cdot \nabla V dx dy 
= \int_{\R \times \R_+} F V dx dy + \int_{\Omega} f \operatorname{tr} V dx. 
\end{align}
Finally, the solution $u$ to \eqref{eq:weakform} is given by 
$u = \operatorname{tr}U$ 
with $U$ being the unique solution of \eqref{eq:extension} with
$F = 0$.

\subsection{Global regularity}

The following Lemma~\ref{lem:regularity} provides additional global regularity in the $x$-variable.  
Although it is a special case of \cite[Lem.~3.2]{FMMS21}, which is valid for any space dimension $d \geq 1$, 
we sketch the proof as it is crucial for this article.  

\begin{lemma}\label{lem:regularity}
Let $f\in C^\infty(\overline{\Omega})$, $F \in C^\infty_0(\R^2)$, and let $U$ solve \eqref{eq:extension}. 
Then, for $t \in [0,1/2)$, there is $C_t > 0$ depending only on $t$ and $\Omega$, such that
\begin{align*}
\int_{\R_+} y^\alpha \norm{\nabla U(\cdot, y)}_{H^t(\Omega)}^2 dy < C_t N^2(F,f)
\end{align*}
with
\begin{align}
\label{eq:CUFf}
N^2(F,f):= 
\|f\|^2_{H^{1}(\Omega)} + \|F\|^2_{L^2_{-\alpha}(\R\times\Rpos)}. 
\end{align}
\end{lemma}
\begin{proof}[Sketch]
The idea is to apply the difference quotient argument from \cite{Savare} only in the $x$-direction.
For $h \in \R$, denote 
$T_h U :=  \eta U_h + (1-\eta)U$, where $U_h(x,y) :=  U(x+h,y)$ and 
$\eta$ is a cut-off function that localizes to a fixed interval
$B_{2\rho}(x_0)$ in the first variable and that is constant in the second variable.

The main result of \cite{Savare} is that estimates for the modulus $\omega(U)$ defined 
by 
\begin{align*}
& \omega(U) :=\\
& \sup_{h \in D\backslash  \{0\}} \frac{b(T_h U,T_h U) - b(U,U)+\int_{\R \times \Rpos} F (T_h U - U)
   +\int_{\Omega}  f \operatorname{tr}( T_h U - U)}{\abs{h}} 
\end{align*}
can be used to derive regularity results in Besov spaces.
Here, $D \subset \R$ denotes a set of admissible directions $h$. 
These directions are chosen such that the function 
$T_h U$ is an admissible test function, i.e., 
$T_h U \in H^1_{\alpha,0}(\R \times \Rpos)$. 
In the present case
this set can easily be characterized as $h \in [-\rho,\rho]$, 
if $B_{4\rho}(x_0) \subset \Omega$ 
or 
$\operatorname{dist}(B_{3\rho}(x_0),\Omega) \geq \rho$. 
In the other cases, we can take $h \in [0,\rho]$, if the right
endpoint of the interval $\Omega$ is in the intersection of $B_{4\rho}(x_0) \cap \Omega$ 
or 
$h \in [-\rho,0]$, if the left
endpoint of the interval $\Omega$ is in the intersection of $B_{4\rho}(x_0) \cap \Omega$.

{\bf Step 1.} (Estimate of $b(\cdot,\cdot)$-terms).
Using support properties of $\eta$ as well as the estimate 
$\norm{U(\cdot,y)-U_h(\cdot,y)}_{L^2(B_{2\rho})} \lesssim \abs{h} \norm{\nabla U(\cdot,y)}_{L^2(B_{3\rho})}$, 
one can deduce 
\begin{align*}
\left| b(T_h U,T_h U)-b(U,U) \right| \lesssim \abs{h} \int_{B_{3\rho}^+}y^\alpha \abs{\nabla U}^2 \, dx \, dy.
\end{align*}

{\bf Step 2.} (Estimate of $F$-integral).
Similarly to the first step, one can estimate
\begin{align*}
 \left| \int_{\R \times \Rpos} F (U-T_h U)\, dx \, dy\right| &\leq \norm{F}_{L^2_{-\alpha}(B_{2\rho}^+)}
 \norm{U-U_h}_{L^2_\alpha(B_{2\rho}^+)}  
 \\ &\lesssim \abs{h} \norm{F}_{L^2_{-\alpha}(B^+_{2\rho})} \norm{\nabla U}_{L^2_{\alpha}(B^+_{3\rho})}.
\end{align*}

{\bf Step 3.} (Estimate of $f$-integral).
With the trace inequality from \cite[Lem.~{3.3}]{KarMel19}, we obtain
\begin{align}\label{eq:SavTmp3}
\left|  \int_{\Omega} f \operatorname{tr}(U-T_h U)\, dx \right| \lesssim \abs{h} \norm{f}_{H^{1}(B_{4\rho})}\norm{\nabla U}_{L^2_\alpha(B_{4\rho}^+)}.
 \end{align}

{\bf Step 4.} (Application of the abstract framework of \cite{Savare}).
We introduce the seminorms $[U]^2:=  \int_{\R \times \Rpos} y^\alpha |\nabla U|^2\,dx dy$. By the 
coercivity of $b(\cdot,\cdot)$ on $H^1_{\alpha,0}(\R \times \Rpos)$ with respect to $[\cdot]^2$ and the abstract 
estimates in \cite[Sec.~{2}]{Savare}, we have 
\begin{align*}
[U-T_h U]^2  &\lesssim  \omega(U) |h| \\
& \lesssim
|h|\|\nabla U\|_{L^2_\alpha(B^+_{4\rho})}\left(
\|\nabla U\|_{L^2_\alpha(B^+_{3\rho})} + \|F\|_{L^2_{-\alpha}(B^+_{2\rho})} + \|f\|_{H^{1}(B_{4\rho})}
\right).
\end{align*}
Employing the {\sl a priori} estimate $
\|\nabla U\|_{L^2_{\alpha}(\R \times \Rpos)} 
\lesssim  \|F\|_{L^2_{-\alpha}(\R \times \Rpos)} + \|f\|_{H^{-s}(\Omega)}$ and
using $\eta \equiv 1$ on $B^+_{\rho}(x_0)$ leads to
\begin{align} 
\label{eq:local-10} 
\int_{B^+_\rho} y^\alpha |\nabla U - \nabla U_h|^2\, dx \, dy &\leq
[U - T_h U]^2 \leq |h| \; N^2(F,f).
\end{align} 
{\bf Step 5a:} ($H^t(\Omega)$--estimate).
Thus far, we only consider one sided difference quotients, i.e., $h \in D$ 
in (\ref{eq:local-10}). The 
restriction $h \in D$ in (\ref{eq:local-10}) 
can be lifted as shown in \cite{FMMS21}.
In the present 1D situation, the key observation is that 
when computing the Aronstein-Slobodecki norm, one can write  
for functions $v$ defined on $\R$
\begin{align*} 
\int_{x\in\R} \int_{|h| \leq h_0} \frac{|v(x+h) - v(x)|^2}{|h|^{1+2\sigma}}\,dh\,dx & = 
\int_{x\in\R} \int_{h=0}^{h_0} 
\frac{|v(x+h) - v(x)|^2}{h^{1+2\sigma}}\,dh\,dx \\
& +  
\int_{x\in\R} \int_{h=0}^{h_0} 
\frac{|v(x-h) - v(x)|^2}{h^{1+2\sigma}}\,dh\,dx,  
\end{align*} 
and a change of variables $x - h  = x'$ in the second integral leads again to a 
one-sided difference quotient. For simplicity of presentation, 
we will therefore assume in the following Step 5b that (\ref{eq:local-10})
holds for all $|h| \leq h_0$, and  we will assume that $B_\rho$ in (\ref{eq:local-10}) can be replaced by $\Omega$; this is possible by 
covering with suitable localizations and using (\ref{eq:local-10}).

{\bf Step 5b:} ($H^t(\Omega)$--estimate).
For $t < 1/2$ and $\widetilde{R}$ large enough (s.t. $\Omega \subset (-\widetilde{R},\widetilde{R})$), we estimate with the Aronstein-Slobodeckij seminorm 
\begin{align*}
\int_{\Rpos} |\nabla U(\cdot,y)|^2_{H^t(\Omega)}\,dy  & \leq 
\int_{\Rpos} \int_{\Omega} \int_{|h| \leq {\widetilde{R}}} \frac{|\nabla U(x+h,y) - \nabla U(x,y)|^2}{|h|^{1+2t}}\,dh\; dx\; dy.  
\end{align*}
The integral in $h$ is split into the range $|h| \leq \varepsilon$ for some fixed $\varepsilon>0$, for which 
\eqref{eq:local-10} can be brought to bear, and $\varepsilon < |h| < \widetilde{R}$, for which a triangle inequality can be used. This gives  the sought estimate. 
\end{proof}

\subsection{Interior regularity}
In the following, we are interested in Caccioppoli type estimates 
that allow to control higher order derivatives by lower order derivatives on slightly enlarged intervals.
\begin{lemma}\label{lem:CaccType} 
Let $B_R \subset \Omega$ be a ball of radius $R$
and $B_{cR}$ the concentrically scaled ball of radius $cR$. 
Let $U$ satisfy \eqref{eq:extension} with given data
$f$ and $F$. 
There is a constant $C_{\rm Cac}$ depending only on $\Omega$, $s$ 
such that 
for every $c\in (0,1)$
\begin{align}\label{eq:Caccioppoli}
 \norm{\nabla \partial_x U}_{L^2_\alpha(B_{cR}^+)} &\leq C_{\rm Cac} \Big( ((1-c)R)^{-1} \norm{\nabla U}_{L^2_\alpha(B_R^+)} \nonumber \\
 & \qquad\qquad+ \norm{\partial_x f}_{L^2(B_R)} + \norm{F}_{L^2_{-\alpha}(B_R^+)} \Big).
\end{align}
With a constant $\gamma>0$ depending only on $s$, $\Omega$, and $c$, it holds
for all $p \in \N$
\begin{align}\label{eq:HighOrderCaccioppoli}
\norm{\nabla \partial^p_x U}_{L^2_\alpha(B_{cR}^+)} \leq&\;  (\gamma p)^p  R^{-p}\norm{\nabla U}_{L^2_\alpha(B_R^+)} \\ \nonumber 
& + \sum_{i=1}^p (\gamma p)^{p-i} R^{i-p}\bigl(\norm{\partial_x^i f}_{L^2(B_R)} + \norm{\partial_x^{i-1} F}_{L^2_{-\alpha}(B_R^+)}\bigr). 
\end{align}
\end{lemma}
\begin{proof}
Estimate \eqref{eq:Caccioppoli} follows from \cite[Lem.~3.4]{FMMS21}, which holds for any space dimension $d$, using difference quotient techniques as previously employed in \cite{FMP21}.

We now show \eqref{eq:HighOrderCaccioppoli}:
As the $x$-derivatives commute with the differential operator in \eqref{eq:extension}, we have that $\partial_x^i U$ solves 
\eqref{eq:extension} on $B^+_R$ with data $\partial_x^i F$ and $\partial_x^i f$ for any $i$.

For given $c>0$, we choose sets $B_{c_iR}$ with $c_i = c + (i-1)\frac{(1-c)}{p}$, which implies 
$c_i < c_{i+1} < 1$ for all $i < p$. Then, we have $c_{i+1}R-c_iR = \frac{(1-c)R}{p}$. Applying  
\eqref{eq:Caccioppoli} to $\partial^{p-1}_x U$ with the sets $B_{c_{1} R}, B_{c_{2} R}$ leads to the estimate
\begin{align*}
 \norm{\nabla  \partial_x^p U}_{L^2_\alpha(B_{c_1R}^+)} &\leq C_{\rm Cac} \Bigl(\frac{p}{(1-c)}R^{-1} \norm{\nabla \partial^{p-1}_x U}_{L^2_\alpha(B_{c_2 R}^+)} \\
&\qquad + \norm{\partial_x^p f}_{L^2(B_R)} + \norm{\partial^{p-1}_x F}_{L^2_{-\alpha}(B_R^+)} \Bigr). 
\end{align*}
Inductively applying \eqref{eq:Caccioppoli}
to control $\norm{\nabla \partial^{p-j+1}_x U}_{L^2_\alpha(B_{c_j R}^+)}$ for $2\leq j\leq p$ with sets $B_{c_{j} R}, B_{c_{j+1} R}$  
provides the claimed estimate with $\gamma = C_{\rm Cac} /(1-c)$. 
\end{proof}

The right-hand side of the Caccioppoli estimate (\ref{eq:Caccioppoli}) suggests that we need to control $R^{-1} \norm{\nabla U}_{L^2_{\alpha}(B_R^+)}$. 
This term is actually small for $R\rightarrow 0$ in the presence of regularity of $U$, which was asserted in 
Lemma~\ref{lem:regularity}.

\begin{lemma}\label{lem:estH1}
Let $ S_{R} :=  \{x \in \Omega \;\colon \; r(x) < R\}$.  
Let $U$ solve \eqref{eq:extension}.
Then, for $t \in [0,1/2)$, there exists $C_{\rm reg} > 0$ depending only on $t$ and $\Omega$ such that, 
with the constant $C_t>0$ from Lemma~\ref{lem:regularity} and $N^2(F,f)$ given by \eqref{eq:CUFf}, we have 
\begin{align}
\label{eq:lem:estH1-5}
R^{-2t} \|\nabla U\|^2_{L^2_{\alpha}(S^+_R)} 
\leq 
\|r^{-t} \nabla U \|^2_{L^2_{\alpha}(\Omega^+)} 
\leq 
C_{\rm reg} C_t N^2(F,f). 
\end{align}
\end{lemma}
\begin{proof}
The first estimate in \eqref{eq:lem:estH1-5} is trivial. 
For the second bound, we start by noting that
Lemma~\ref{lem:regularity} provides the global regularity
$\displaystyle \int_{\Rpos} y^\alpha \norm{\nabla U(\cdot, y)}_{H^t(\Omega)}^2 dy \leq C_t N^2(F,f). 
$
For $t \in [0,1/2)$ and any $v \in H^t(\Omega)$, we have by, e.g., \cite[Thm.~{1.4.4.3}]{Grisvard} the embedding result 
$\|r^{-t} v\|_{L^2(\Omega)} \leq C_{\mathrm{reg}}
\|v\|_{H^t(\Omega)}$. Applying this embedding to $\nabla U(\cdot,y)$,
multiplying by $y^\alpha$, and integrating in $y$ yields the claimed estimate. 
\end{proof}

With the Caccioppoli estimate, we obtain estimates for the derivatives.

\begin{lemma}\label{lem:summation}
Let $U$ solve \eqref{eq:extension} with data $f$, $F$ satisfying 
for some $C_f$, $C_F$, $\gamma >0$
\begin{align*}
\forall p \in \N_{0} \colon \quad 
\norm{\partial_x^p f}_{L^2(\Omega)} \leq C_f \gamma^p p^p, \qquad \norm{\partial_x^p F}_{L^2_{-\alpha}(\Omega^+)} \leq C_F \gamma^p p^p.
\end{align*}
Then, there is $\tilde \gamma > 0$ (depending on $\gamma$, $s$, $\Omega$) 
and, for every $\varepsilon \in (0,1)$, $t \in [0,1/2)$ a constant $C_\varepsilon$ 
(depending only on $\varepsilon$, $t$, $\Omega$) 
such that for all $p \in \N_{0}$, 
we have  
\begin{align*}
\norm{r^{p-t+\varepsilon} \nabla \partial_x^p  U}_{L^2_\alpha(\Omega^+)} \leq C_{\varepsilon}\tilde \gamma^p p! (C_{f} +C_F + N(F,f)).
\end{align*}
\end{lemma}
\begin{proof}
The case $p = 0$ follows immediately from Lemma~\ref{lem:estH1}.
For $p \ge 1$ and for any ball $B_R\subset \Omega$ and $c>0$, 
Lemma~\ref{lem:CaccType} combined with Lemma~\ref{lem:estH1} gives 
\begin{align}
\label{eq:lem:summation-10}
\norm{\nabla \partial_x^p U}_{L^2_\alpha(B_{cR}^+)} &\lesssim \gamma^p p^p R^{t-p}N(F,f) \\ 
\nonumber 
&\quad + 
\sum_{i=1}^p \gamma^{p-i}  p^{p-i} R^{i-p} \Bigl( \norm{\partial_x^i f}_{L^2(B_R)}+ \norm{\partial_x^{i-1} F}_{L^2_{-\alpha}(B_R^+)}\Bigr). 
\end{align}
In order to obtain an estimate for the $L^2$-norm over $\Omega = (a,b)$, 
we dyadically cover $\Omega \subset \bigcup_{i \in \N} B_{c r_i} (x_i)$ 
with intervals $B_{c r_i} (x_i) \subset \Omega$, 
where the points $x_i$ run through 
the set $\{a+(b-a)2^{-j}\colon j \in \N\} \cup \{b - (b-a) 2^{-j}\colon j \in \N\},$ 
$c \in (1/2,1)$ is fixed and $r_i = \dist(x_i,\partial\Omega)$. A geometric
series argument gives $\sum_i r_i^\varepsilon  <\infty$ for any 
chosen $\varepsilon$.  
Using the assumption on the data $f$, $F$ and (\ref{eq:lem:summation-10}), we obtain using $t < 1/2$
\begin{align}
\nonumber 
\norm{r^{p}\nabla \partial_x^p U}_{L^2_\alpha(B_{cr_i}(x_i)^+)} & \lesssim r_i^p\norm{\nabla \partial_x^p U}_{L^2_\alpha(B_{cr_i}(x_i)^+)}   \\
\label{eq:estD2}
& \lesssim \tilde \gamma^p p^p r_i^t(C_f+C_F+N(F,f)) 
\end{align}
for suitable $\tilde \gamma$.
Therefore, we have for fixed $\varepsilon>0$
\begin{align*}
& \norm{r^{p-t+\varepsilon}(\nabla \partial_x^p U)}_{L^2(\Omega^+)}^2 \lesssim 
 \sum_{i \in \N} r_i^{2p-2t+2\varepsilon} \norm{\nabla \partial_x^p U}_{L^2(B_{c r_i} (x_i)^+)}^2 \\
 &\quad \stackrel{\eqref{eq:estD2}}{\lesssim} \tilde \gamma^{2p} p^{2p}
 \sum_{i \in \N} r_i^{2\varepsilon} (C_f + C_{F} +N(F,f))^2 
 = C_{\varepsilon}^2\tilde \gamma^{2p} p^{2p} (C_f + C_{F}+N(F,f))^2,
\end{align*}
which finishes the proof after noting $p^p \leq p! e^p$ and adjusting $\tilde \gamma$. 
\end{proof}

Taking traces, we obtain weighted analytic estimates for the fractional Laplacian:
\begin{proof}[of Theorem~\ref{thm:analytic-regularity}]
\cite[Lem.~3.7]{KarMel19} gives
\begin{align*}
 \abs{V(x,0)}^2 \lesssim \norm{V(x,\cdot)}_{L^2_\alpha(\R_+)}^{1-\alpha}\norm{\partial_y V(x,\cdot)}_{L^2_\alpha(\R_+)}^{1+\alpha}
+\norm{V(x,\cdot)}^2_{L^2_\alpha(\R_+)}.
\end{align*}
Using this trace estimate with $V = \partial^p_x U$, additionally multiplying with $r^{2p-1-2s+2\varepsilon}$, and using $\alpha=1-2s$ provides 
\begin{align*}
  & r^{2p-1-2s+2\varepsilon}\abs{\partial_x^pU(x,0)}^2  \lesssim 
\|r^{p-1/2-s+\varepsilon} \partial^p_x U(x,\cdot)\|^2_{L^2_\alpha(\Rpos)} \\
&\qquad \qquad + \norm{r^{p-3/2+\varepsilon}\partial_x^pU(x,\cdot)}_{L^2_\alpha(\R_+)}^{1-\alpha}
 \norm{r^{p-1/2+\varepsilon}\partial_y \partial_x^p U(x,\cdot)}_{L^2_\alpha(\R_+)}^{1+\alpha}. 
\end{align*}
Integration over $\Omega$ gives in view of $u(x) = U(x,0)$ 
\begin{align*}
& \norm{r^{p-1/2-s+\varepsilon} \partial_x^p  u}_{L^2(\Omega)}^2 
\lesssim 
\|r^{p-1/2-s+\varepsilon} \partial^p_x U(x,\cdot)\|^2_{L^2_\alpha(\Omega^+)} \\
& \qquad \qquad + \norm{r^{p-3/2+\varepsilon}\partial_x^p U}_{L^2_\alpha(\Omega^+)}^{1-\alpha}
\norm{r^{p-1/2+\varepsilon}\partial_y \partial_x^p U }_{L^2_\alpha(\Omega^+)}^{1+\alpha}.
\end{align*}
Note that for $p \ge 1$, we have $|\partial^p_x U| \leq |\nabla \partial^{p-1}_x U|$ and  
 $|\partial_y \partial^p_x U| \leq |\nabla \partial^{p}_x U|$.  
Applying Lemma~\ref{lem:summation} with $t=1/2-\varepsilon/2$ and $\varepsilon/2$ instead of $\varepsilon$ therein 
for the two terms with weights $r^{p-3/2+\varepsilon}$ and $r^{p-1/2+\varepsilon}$ and 
$t = \max(0,s-1/2)$ for the term with the weight $r^{p-1/2-s+\varepsilon}$
provides the desired estimate. 

The statement $u \in {\mathcal B}^1_\beta$ follows by definition. The assertion $u \in C(\overline{\Omega})$ is implied by 
the observation $u \in C^\infty(\Omega)$ together with $u \in L^2(\Omega)$ and $r^{-1/2-s+\varepsilon} u^\prime  \in L^2(\Omega)$. 
\end{proof}

\subsection{Exponential convergence of $hp$-FEM}
\label{sec:ExConvhp}
For $\beta' \in [0,1)$, it is convenient to introduce the norm $\|\cdot\|_{H^1_{\beta'}(\Omega)}$ by 
$$
\|v\|^2_{H^1_{\beta'}(\Omega)}:= \|r^{\beta'} v^\prime\|^2_{L^2(\Omega)} + \|r^{\beta'-1} v\|^2_{L^2(\Omega)}. 
$$

\begin{lemma} 
\label{lemma:weighted-embedding}
Let $\beta' \in [0,1)$, $\sigma \in (0,1-\beta')$. 
Then, there is $C_{\beta',\sigma}> 0$ such that 
\begin{equation}
\|v\|_{\widetilde{H}^\sigma(\Omega)} \leq C_{\beta',\sigma} \left[ \|r^{\beta'} v^\prime\|_{L^2(\Omega)} + \|r^{\beta'-1} v\|_{L^2(\Omega)}\right] 
\end{equation}
for all $v$ such that right-hand side is finite. 
\end{lemma}
\begin{proof} 
For two continuously embedded Banach spaces $X_1 \subset X_0$ 
and for $v\in X_0$, $t>0$, we define the $K$-functional by 
$K(v,t):= \inf_{w \in X_1} \|v - w\|_{X_0} + t\|w\|_{X_1}$. 
For $\theta \in (0,1)$ and fine index $q\in [1,\infty]$,
the interpolation spaces (e.g. \cite[Lecture 22]{tartar07})
$X_{\theta,q}:= (X_0,X_1)_{\theta,q}$ 
are given by the norm
$$
\|v\|^q_{X_{\theta,q}}
:= 
\int_{t=0}^\infty \left(t^{-\theta} K( v,t)\right)^q\frac{dt}{t}, \; q \in [1,\infty), 
\quad 
\|v\|_{X_{\theta,\infty}}:= \sup_{t \in( 0, \infty)} t^{-\theta} K(v,t).   
$$
We use the fact that, since $X_1 \subset X_0$, 
replacing the integration and the supremum limit $\infty$ by a finite number $T$ 
leads to an equivalent norm, \cite[Chap.~6, Sec.~7]{devore93}. 
Let $S_t:=\{x \in \Omega\,|\, r(x) < t\}$
denote a $t$-neighborhood of $\partial\Omega$.  
For each $t > 0$ sufficiently small, we may choose $\chi_t \in C^\infty_0(\R)$ such that 
$\chi_t \equiv 0$ on $S_{t/2}$
and 
$\chi_t \equiv 1$ on $\Omega \setminus S_t$ as well as 
$\|\chi_t^{(j)}\|_{L^\infty(\R)} \leq C t^{-j}$, $j \in \{0,1\}$. 
Decomposing $v = \chi_t v +  (1-\chi_t) v$, we have $\chi_t v \in H^1_0(\Omega)$ 
and $(1 - \chi_t) v \in L^2(\Omega)$. 
A calculation reveals 
\begin{align*}
\|(\chi_t v)^\prime\|_{L^2(\Omega)} & 
\leq C t^{-\beta'} \|v\|_{H^1_{\beta'}(\Omega)}, \\
\|(1-\chi_t) v\|_{L^2(\Omega)} & \leq C \|v\|_{L^2(S_t)} \leq t^{1-\beta'} \|r^{\beta'-1} v\|_{L^2(\Omega)}. 
\end{align*}
This implies for $X_0 = L^2(\Omega)$, $X_1 = H^1_0(\Omega)$ that 
$K(v,t) \leq C t^{1-\beta'} \|v\|_{H^1_{\beta'}(\Omega)}$. For 
the Besov space $\widetilde B^{1-\beta'}_{2,\infty}(\Omega) := (L^2(\Omega), H^1_0(\Omega) )_{1-\beta',\infty}$, we obtain  
$\|v\|_{\widetilde{B}^{1-\beta'}_{2,\infty}(\Omega)} \leq C \|v\|_{H^1_{\beta'}(\Omega)}$. 
We conclude the proof by noting 
\begin{equation*}
\widetilde H^\sigma(\Omega) \stackrel{\scriptsize\cite{tartar07}}{=}
(L^2(\Omega), H^1_0(\Omega))_{\sigma,2} \stackrel{\sigma < 1-\beta'}{\subset }
(L^2(\Omega), H^1_0(\Omega))_{1-\beta',\infty}  = \widetilde{B}^{1-\beta'}_{2,\infty}(\Omega). 
\tag*{\hbox{\rlap{$\sqcap$}$\sqcup$}} 
\end{equation*}
\end{proof}

\begin{lemma}
\label{lemma:approx-on-small-element} 
Let $\beta' \in [0,1)$, $\varepsilon > 0$. 
Then, there is $C_{\beta',\varepsilon} > 0$ such that the following holds: 
For $\widehat{v} \in C([0,1])$ let $I\widehat{v}$ be the linear interpolant in the endpoints $0$, $1$. Then, provided
the right-hand side is finite, there holds for $e:= \widehat{v} - I \widehat{v}$ 
\begin{align*}
\|x^{\beta'-1} e\|_{L^2(0,1)} + 
\|x^{\beta'} e^\prime\|_{L^2(0,1)} 
\leq C_{\beta',\varepsilon} \|x^{\min\{\beta' + 1, 3/2-\varepsilon\}} \widehat{v}^{\prime\prime}\|_{L^2(0,1)}.  
\end{align*}
\end{lemma}
\begin{proof}
\emph{Step 1:}
Let $\widetilde \pi_1 \widehat v \in {\mathcal P}_1$ be the linear interpolant of $\widehat v$ in the points $1/2$ and $1$. From 
\cite[Lemma~{15}, (A.4)]{banjai-melenk-nochetto-otarola-salgado-schwab19}, we get for any $\alpha >-1$ 
\begin{align}
\label{eq:focm-lemma15}
\|x^{\alpha/2} (\widehat v - \widetilde{\pi}_1 \widehat v)^\prime\|_{L^2(0,1)} \leq C_{\alpha}  \|x^{\alpha/2+1} \widehat v^{\prime\prime}\|_{L^2(0,1)}. 
\end{align}
A maximum norm estimate is obtained from 
\begin{align*}
& |(\widehat v - \widetilde{\pi}_1 \widehat v)(x)|   \leq  \int_{x}^1 |(\widehat v - \widetilde{\pi}_1 \widehat v)^\prime(t)|\,dt  \\
& \leq \sqrt{ \int_{0}^1 t^{-1+2 \varepsilon} \, dt \int_0^1 t^{1-2\varepsilon} |(\widehat v - \widetilde{\pi}_1 \widehat v)^\prime(t)|^2\,dt} 
\stackrel{(\ref{eq:focm-lemma15})}{\leq} C_\varepsilon \|x^{3/2-\varepsilon} \widehat v^{\prime\prime}\|_{L^2(0,1)}. 
\end{align*}

\emph{Step 2 ($\beta' < 1/2$):} Abbreviate $e:= \widehat v - I \widehat v$ and note $e(x) = \int_0^x e^\prime(t)\,dt$. 
For $\beta' \in [0,1/2)$, Hardy's inequality \cite[Chap.~2, Thm.~{3.1}]{devore93} is applicable and yields 
$$
\|x^{\beta'-1} e\|_{L^2(0,1)} \leq C \|x^{\beta'} e^\prime\|_{L^2(0,1)}. 
$$
We estimate with $w   = \widetilde{\pi}_1 \widehat v \in {\mathcal P}_1$ 
\begin{align*}
\|x^{\beta'} e^\prime\|_{L^2(0,1)} &\leq 
\|x^{\beta'} (\widehat v - w)^\prime \|_{L^2(0,1)} + 
\|x^{\beta'} (I(\widehat v - w))^\prime \|_{L^2(0,1)} \\
&\stackrel{\text{${\mathcal P}_1$ finite dimensional}}{\lesssim} 
\|x^{\beta'} (\widehat v - w)^\prime \|_{L^2(0,1)} + \|\widehat v - w\|_{L^\infty(0,1)} \\
&\stackrel{\text{{Step~1}, $w = \widetilde{\pi}_1 \widehat v$}}{\lesssim} 
\|x^{\min\{\beta'+1, 3/2-\varepsilon\}} \widehat v^{\prime\prime}\|_{L^2(0,1)}. 
\end{align*}
\emph{Step 3 ($\beta' >  1/2$):} From the representation $e(x) = \int_0^x e^\prime(t)\,dt$ we get 
for any $\alpha \in [0,1/2)$ by the Cauchy-Schwarz inequality 
$ |e(x)| \leq C_\alpha x^{1/2-\alpha} \|x^\alpha e^\prime\|_{L^2(0,1)}$. Hence, for $\alpha$ sufficiently close to $1/2$,
$$
\|x^{\beta'-1} e\|_{L^2(0,1)} \lesssim \sqrt{\int_0^1 x^{2\beta'-2+1-2\alpha}\,dx} \; \|x^\alpha e^\prime\|_{L^2(0,1)} 
\lesssim \|x^\alpha e^\prime\|_{L^2(0,1)}. 
$$
We conclude, since $\alpha < 1/2 < \beta'$, that 
$\|x^{\beta'-1} e\|_{L^2(0,1)} + \|x^{\beta'} e^\prime\|_{L^2(0,1)} \lesssim \|x^\alpha e^\prime\|_{L^2(0,1)}$.  
Applying Step~2 with $\alpha$ taking the role of $\beta'$ there, we get, by selecting $\alpha$ sufficiently close to $1/2$ 
\begin{align*}
\|x^\alpha e^\prime\|_{L^2(0,1)} \lesssim \|x^{3/2-\varepsilon} u^{\prime\prime}\|_{L^2(0,1)}. 
\end{align*}
\emph{Step 4 ($\beta' = 1/2$):} Given $\varepsilon>0$, we note for $\varepsilon'>0$ sufficiently small that 
$\min\{\beta'+1,3/2-\varepsilon\} = 3/2-\varepsilon = \min\{\beta'-\varepsilon' +1,3/2-\varepsilon\}$ so that 
we obtain the result by applying Step~2 to the choice $\beta' = 1/2-\varepsilon'$. 
\end{proof}
\begin{proof}[of Theorem~\ref{thm:hp-approx}]
Since the bilinear form $a(\cdot,\cdot)$ is elliptic on $\widetilde{H}^s(\Omega)$, 
by C\'ea's lemma it suffices to construct a function 
$v \in S^{p,1}_0(\mathcal{T}^{L}_{geo,\sigma})$ that satisfies the stated error bound. 
By Theorem~\ref{thm:analytic-regularity}, we have 
$u \in {\mathcal B}^1_\beta(\Omega)$ for $\beta = 1/2-s + \varepsilon$, $\varepsilon > 0$ arbitrary and $u \in C(\overline{\Omega})$. 
We approximate $u \in C(\overline{\Omega})$ by a $v \in S^{p,1}({\mathcal T}^{L}_{geo,\sigma})$ defined as 
the linear interpolant in the elements abutting $\partial\Omega$ and the Gauss-Lobatto interpolant of degree $p$ on all other elements. 
Selecting $\beta':= 1-s-\varepsilon'$ for $\varepsilon'$ sufficiently small, Lemma~\ref{lemma:weighted-embedding} shows 
$
\|u - v\|_{\widetilde{H}^s(\Omega)} \leq C \|u - v\|_{H^1_{\beta'}(\Omega)}, 
$
which can be estimated by summing elementwise error contributions. 
For the approximation on the elements abutting $\partial\Omega$, we observe that, for sufficiently small $\varepsilon$, 
we can select $\varepsilon'$ to ensure 
\begin{equation}
3/2-s+\varepsilon = \beta+1 < \min\{\beta'+1,3/2-\varepsilon\} = \min\{2-s-\varepsilon', 3/2-\varepsilon\}. 
\end{equation}
In Lemma~\ref{lemma:approx-on-small-element}, 
we may therefore replace $x^{\min\{\beta'+1,3/2-\varepsilon\}}$ by $x^{\beta+1}$. 
Combining then Lemma~\ref{lemma:approx-on-small-element} with a scaling argument, 
we get for $T_i$ with $\overline{T}_i \cap \partial\Omega \ne \emptyset$ 
\begin{align*}
\|r^{\beta'-1} (u - v) \|_{L^2(T_i)} + 
\|r^{\beta'} (u - v)^\prime \|_{L^2(T_i)} \leq C h_i^{1-\beta-(1-\beta')} \|r^{\beta+1} u^{\prime\prime}\|_{L^2(T_i)}.  
\end{align*}
For the remaining elements $T_i$ with $\overline{T}_i \cap \partial\Omega = \emptyset$, 
we use $u \in {\mathcal B}^1_\beta(\Omega)$ and get following 
\cite[Thm.~{3.13}]{ApeMel15} with a scaling argument 
\begin{align*}
\|r^{\beta'-1} (u - v) \|_{L^2(T_i)} + 
\|r^{\beta'} (u - v)^\prime \|_{L^2(T_i)} \leq C h_i^{1-\beta - (1-\beta')} e^{-b p} 
\end{align*}
for some $b > 0$ independent of $p$. 
Noting that $\beta' - \beta = 1/2-\varepsilon - \varepsilon' > 0$, we obtain from a geometric series argument
by summation over all elements and using $h_i = O(\sigma^L)$ for the elements abutting on $\partial\Omega$ that 
$$
\|u - v\|_{H^1_{\beta'}(\Omega)} \lesssim \sigma^{L(\beta'-\beta)} + e^{-b p} = \sigma^{L(1/2-\varepsilon-\varepsilon')} + e^{-bp}. 
$$
The proof of Theorem~\ref{thm:hp-approx} is completed by suitably adjusting $\varepsilon$.  
\end{proof}

\section{Numerical example}

%
On the interval $(-1,1)$ with $f = 1$, the exact solution to (\ref{eq:weakform}) is 
$\displaystyle 
u(x) = 2^{-2s}\sqrt{\pi}(\Gamma(s+1/2)\Gamma(1+s))^{-1} (1-x^2)^s.
$
The singularity of $u(x) \sim {\rm dist}(x,\partial\Omega)^s$ is generic for the Dirichlet problem
of the integral fractional Laplacean, also in higher dimensions, \cite{ros-oton-serra14}.
We employ the geometric mesh ${\mathcal T}^L_{geo,\sigma}$ that is 
graded towards $\pm 1$ with grading factor $\sigma = 0.6$ and $L$ levels of refinement. 
We utilize two $hp$-FEM spaces: a) the space $S^{L,1}_0({\mathcal T}^{L}_{geo,\sigma})$ of piecewise polynmomials
of degree $p = L$ and b) the subspace $\widetilde S^L \subset S^{L,1}_0({\mathcal T}^{L}_{geo,\sigma})$ obtained 
by lowering the polynomial degree from $p = L$ to $1$ in the two elements touching $\partial\Omega$.  
The $hp$-FEM approximation is computed in {\sc Matlab} by using an implementation from \cite{DadicBacc}.

 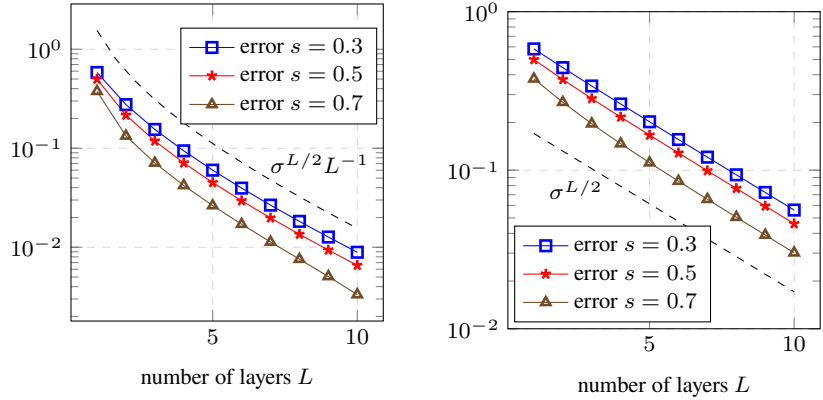
\begin{figure}[t]
 \begin{minipage}{.49\linewidth}
 \begin{center}

\begin{tikzpicture}
\begin{semilogyaxis}[
    width=\textwidth, height=5.8cm,     
    grid = major,
    grid style={dashed, gray!30},
    axis background/.style={fill=white},
    xlabel={number of layers $L$}, 
    legend style={at={(0.35,0.62)},anchor=south west,font=\footnotesize},
    xtick = {},
    ytick = {},
    ]
    
    \addplot+[solid,mark=square,mark size=2pt,mark options={line width=1.0pt}] table    
    [
    x=L,
    y=err,
    col sep=comma
    ]{dataS03.txt};    
    \addlegendentry{error $s = 0.3$}

    \addplot+[solid,mark=star,mark size=2pt,mark options={line width=1.0pt}] table    
    [
    x=L,
    y=err,
    col sep=comma
    ]{dataS05.txt};    
    \addlegendentry{error $s = 0.5$}

    \addplot+[solid,mark=triangle,mark size=2pt,mark options={line width=1.0pt}] table    
    [
    x=L,
    y=err,
    col sep=comma
    ]{dataS07.txt};    
    \addlegendentry{error $s = 0.7$}
%

    \addplot [black,dashed ] expression [domain=1:10, samples=15] {2*0.6^(x/2)/x} node [below,xshift=-0.5cm,yshift=1.1cm] {$\sigma^{L/2} L^{-1}$};

 \end{semilogyaxis} 
\end{tikzpicture}
 
 \end{center}
 \end{minipage}
  \begin{minipage}{.49\linewidth}
 \begin{center}

\begin{tikzpicture}
\begin{semilogyaxis}[
    width=\textwidth, height=5.8cm,     
    ymin = 0.01,
    ymax = 1,
    grid = major,
    grid style={dashed, gray!30},
    axis background/.style={fill=white},
    xlabel={number of layers $L$}, 
    legend style={at={(0.02,0.02)},anchor=south west,font=\footnotesize},
    xtick = {},
    ytick = {},
    ]
    
    \addplot+[solid,mark=square,mark size=2pt,mark options={line width=1.0pt}] table    
    [
    x=L,
    y=err,
    col sep=comma
    ]{dataS03p1.txt};    
    \addlegendentry{error $s = 0.3$}

    \addplot+[solid,mark=star,mark size=2pt,mark options={line width=1.0pt}] table    
    [
    x=L,
    y=err,
    col sep=comma
    ]{dataS05p1.txt};    
    \addlegendentry{error $s = 0.5$}

    \addplot+[solid,mark=triangle,mark size=2pt,mark options={line width=1.0pt}] table    
    [
    x=L,
    y=err,
    col sep=comma
    ]{dataS07p1.txt};    
    \addlegendentry{error $s = 0.7$}
%

    \addplot [black,dashed ] expression [domain=1:10, samples=15] {0.22*0.6^(x/2)} node [below,xshift=-2.9cm,yshift=1.6cm] {$\sigma^{L/2}$};

 \end{semilogyaxis} 
\end{tikzpicture}
 
 \end{center}
 \end{minipage}
\caption{Exponential energy norm error convergence of 
             $hp$-FEM on geometric mesh with grading factor $\sigma = 0.6$ 
for $s \in \{0.3, 0.5, 0.7\}$, $\Omega = (-1,1)$, $f = 1$. Left: $hp$-FEM 
based on $S^{p,1}_0({\mathcal T}^{L}_{geo,\sigma})$ with $p = L$. Right: $hp$-FEM 
based on subspace $\widetilde{S}^L \subset S^{L,1}_0({\mathcal T}^{L}_{geo,\sigma})$.} 
\label{fig:convergence}
 \end{figure}
 In Fig.~\ref{fig:convergence},
 the energy norm error $\sqrt{a(u - u_N,u - u_N)} = \sqrt{a(u,u) - a(u_N,u_N)}$ 
between the exact solution $u\in \widetilde{H}^s(\Omega)$ 
and the $hp$-FEM approximation $u_N$ is plotted versus $L$. 
Note that $a(v,v) \sim \|v\|^2_{\widetilde{H}^s(\Omega)}$. 
The left panel shows the performance of $hp$-FEM based on 
$S^{L,1}_0({\mathcal T}^{L}_{geo,\sigma})$ whereas the right panel that of $hp$-FEM based on 
$\widetilde S^L$. 
As predicted by Theorem~\ref{thm:hp-approx}, 
we observe exponential convergence with respect to the number of layers $L$.
In fact, the 
convergence is close to ${\mathcal O}(\sigma^{L/2} L^{-1})$. The additional factor $L^{-1}$ is due to the fact that we approximate
by polynomials of degree $p = L$ on the boundary elements, whereas the proof of Theorem~\ref{thm:hp-approx} employed 
the linear interpolant on these elements, i.e., an approximation in $\widetilde S^L$. The right panel of Fig.~\ref{fig:convergence} shows the convergence behavior ${\mathcal O}(\sigma^{L/2})$ predicted by Theorem~\ref{thm:hp-approx}. 

\begin{acknowledgement}
The research of JMM was supported by the Austrian Science Fund (FWF) 
project F 65.
\end{acknowledgement}
\bibliography{bibliography}{}
\bibliographystyle{spmpsci}
\end{document}